\documentclass{lms}
\usepackage{amssymb,amsmath}

\newtheorem{theorem}{Theorem}[section]

\newtheorem{lemma}[theorem]{Lemma}

\newnumbered{remarks}{Remarks}

\newcommand{\CC}{{\mathbb C}}
\newcommand{\DD}{{\mathbb D}}

\newcommand{\TT}{{\mathbb T}}
\newcommand{\cB}{{\mathcal B}}
\newcommand{\cD}{{\mathcal D}}
\newcommand{\cH}{{\mathcal H}}
\newcommand{\cM}{{\mathcal M}}
\DeclareMathOperator{\aut}{\rm Aut}
\DeclareMathOperator{\hol}{\rm Hol}

\DeclareMathOperator{\bmoa}{\rm BMOA}
\DeclareMathOperator{\vmoa}{\rm VMOA}

\title[Gleason--Kahane--\.Zelazko theorem for modules]{A Gleason--Kahane--\.Zelazko theorem for modules and applications to holomorphic function spaces}

\author{Javad Mashreghi and Thomas Ransford}

\classno{46H05 (primary),  30H10, 46H40, 47B33 (secondary)}

\extraline{First  author supported by NSERC.
Second author  supported by NSERC and the Canada research chairs program.}

\begin{document}
\maketitle

\begin{abstract}
We generalize the Gleason--Kahane--\.Zelazko theorem to modules. 
As an application, we show that every linear functional on a Hardy space 
that is non-zero on outer functions is a multiple of a point evaluation. 
A further consequence is that every  linear endomorphism of a Hardy space 
that maps outer functions to nowhere-zero functions 
is a weighted composition operator. 
In neither case is continuity assumed. 
We also consider some extensions to other function spaces, 
including the Bergman, Dirichlet and Besov  spaces, the little Bloch space  and $\vmoa$.
\end{abstract}

\maketitle

\section{A Gleason--Kahane--\.Zelazko theorem for modules}\label{S:GKZ}

The following result, often known as the Gleason--Kahane--\.Zelazko (GKZ) theorem,
char\-acter\-izes multiplicativity of linear functionals on Banach algebras. 
For commutative algebras it was obtained independently 
by Gleason \cite{Gl67} and by  Kahane and \.Zelazko \cite{KZ68}.
Subsequently \.Zelazko \cite{Ze68}  extended it to the  non-commutative case.
The original proofs used results about entire functions. 
An elementary proof can be found in \cite{RS81}.
Note that continuity is not assumed.

\begin{theorem}\label{T:GKZ}
Let $A$ be a complex unital Banach algebra, 
and let $\Lambda:A\to\CC$ be a linear functional such that
$\Lambda(1)=1$ and $\Lambda(a)\ne0$ for all invertible elements $a\in A$. 
Then $\Lambda(ab)=\Lambda(a)\Lambda(b)$ for all $a,b\in A$.
\end{theorem}

We extend the GKZ-theorem to $A$-modules, as follows.

\begin{theorem}\label{T:GKZmodule}
Let $A$ be a  complex unital Banach algebra, let $M$ be a left $A$-module,
and let $S$ be a non-empty subset of $M$ satisfying the following conditions:
\begin{itemize}
\setlength{\itemindent}{15pt}
\item[(S1)] $S$ generates $M$ as an $A$-module;
\item[(S2)] if $a\in A$ is invertible and $s\in S$, then $a s\in S$;
\item[(S3)] for all $s_1,s_2\in S$, there exist $a_1,a_2\in A$ such that 
$a_j S\subset S~(j=1,2)$ and $a_1 s_1=a_2 s_2$.
\end{itemize}
Let $\Lambda:M\to\CC$ be a linear functional such that 
$\Lambda(s)\ne0$ for all $s\in S$. 
Then there exists a unique character $\chi$ on $A$ such that
\begin{equation}\label{E:lambdachi}
\Lambda(a m)=\chi(a)\Lambda(m) \qquad(a\in A,~m\in M).
\end{equation}
\end{theorem}

\begin{remarks*}
(i) The $A$-module $M$ is not assumed to carry any topological structure.

(ii) This result contains the GKZ-theorem as a special case
(take $M:=A$ and $S:=A^{-1}$). 
Note, however, that the GKZ-theorem is used in its proof.
\end{remarks*}

\begin{proof}
Uniqueness of $\chi$ is clear. 
Indeed, fixing any $s\in S$, by \eqref{E:lambdachi} we must have
\[
\chi(a)=\Lambda(a s)/\Lambda(s) \qquad(a\in A).
\]

To prove existence, we begin by deriving inspiration from this last equation. 
Given $s\in S$,  define $\chi_s:A\to\CC$ by
\[
\chi_s(a):=\Lambda(a s)/\Lambda(s) \qquad(a\in A).
\]
Clearly $\chi_s$ is a linear functional on $A$ satisfying $\chi_s(1)=1$, 
and from property~(S2) we have $\chi_s(a)\ne0$ for all invertible $a\in A$.
By Theorem~\ref{T:GKZ}, it follows that $\chi_s$ is a character on $A$.

Next, given $s_1,s_2\in S$,  property~(S3)
yields the existence of elements $a_1,a_2\in A$ such that 
$a_j S\subset S~(j=1,2)$ and $a_1 s_1=a_2 s_2$. 
Then $\Lambda(a_1 s_1)=\Lambda(a_2 s_2)$, whence
\begin{equation}\label{E:chieqn1}
\chi_{s_1}(a_1)\Lambda(s_1)=\chi_{s_2}(a_2)\Lambda(s_2).
\end{equation}
Equally, for each $a\in A$, we have $aa_1 s_1=aa_2 s_2$, whence
\begin{equation}\label{E:chieqn2}
\chi_{s_1}(aa_1)\Lambda(s_1)=\chi_{s_2}(aa_2)\Lambda(s_2).
\end{equation}
Now both sides of \eqref{E:chieqn1} are non-zero, because $a_j s_j\in S$.
Thus we may divide \eqref{E:chieqn2} by \eqref{E:chieqn1} to obtain
\[
\chi_{s_1}(a)=\chi_{s_2}(a).
\]
In other words, $\chi_s$ is independent of $s$. 
Let us call it simply $\chi$. 
Note that we then have
\[
\Lambda(a s)=\chi(a)\Lambda(s) \qquad(a\in A,~ s\in S).
\]

Finally, let $a\in A$ and $m\in M$. 
By property~(S1), there exist $a_1,\dots,a_n\in A$ and $s_1,\dots,s_n\in S$ 
such that $m=\sum_1^na_j s_j$. Then we have
\[
\Lambda(a m)
=\sum_1^n\Lambda(aa_j s_j)
=\sum_1^n \chi(aa_j)\Lambda(s_j)
=\chi(a)\sum_1^n \chi(a_j)\Lambda(s_j)
=\chi(a)\Lambda(m),
\]
which gives \eqref{E:lambdachi}.
\end{proof}

\section{Applications to Hardy spaces}\label{S:Hardy}

Let $\DD$ denote the open unit disk and $\TT$ denote the unit circle. 
We write $\hol(\DD)$ for the space of holomorphic functions on $\DD$.
The Hardy spaces on $\DD$ are defined as follows:
\begin{align*}
H^p
&:=\Bigl\{f\in\hol(\DD):\sup_{r<1}\int_0^{2\pi}|f(re^{i\theta})|^p\,d\theta<\infty\Bigr\}
\quad(0<p<\infty),\\
H^\infty
&:=\Bigl\{f\in\hol(\DD):\sup_{z\in\DD}|f(z)|<\infty\Bigr\}.
\end{align*}
We say that $g\in\hol(\DD)$ is {\em outer} if there exists
$G:\TT\to[0,\infty)$ with $\log G\in L^1(\TT)$ such that
\[
g(z)=\exp\Bigl(\int_0^{2\pi}\frac{e^{i\theta}+z}{e^{i\theta}-z}
\log G(e^{i\theta})\,\frac{d\theta}{2\pi}\Bigr)\
\quad(z\in\DD).
\]
In this case, $g\in H^p$ if and only if $G\in L^p(\TT)$. 
For background on Hardy spaces, we refer to \cite{Du70}.

\begin{theorem}\label{T:Hplinfun}
Let $0< p\le \infty$ and let $\Lambda:H^p\to\CC$ be a linear functional  
such that $\Lambda(g)\ne0$ for all outer functions $g\in H^p$.
Then there exist $w\in \DD$ and $c\in\CC\setminus\{0\}$ such that
\[
\Lambda(f)=cf(w) \qquad(f\in H^p).
\]
\end{theorem}

Note that  continuity of $\Lambda$  is not assumed.

\begin{proof}
We apply Theorem~\ref{T:GKZmodule} with $M:=H^p$ and $A:=H^\infty$, 
taking  $S$ to be the set of outer functions in $H^p$.
Property~(S1) holds because, 
by the canonical  factorization theorem \cite[Theorem~2.8]{Du70},
every function in $H^p$ can be expressed in an essentially unique way 
as the product of an inner function (which is in $H^\infty$) 
and an outer function in $H^p$.
Property~(S2) holds because every invertible function in $h\in H^\infty$ is outer: 
indeed, multiplying together  the inner-outer factorizations of $h$ and $1/h$, 
we obtain a factorization of $1$, and by uniqueness 
it follows  that the inner factors of $h$ and $1/h$  must both be $1$.
Property~(S3) holds because every outer function can be represented 
as the quotient of two bounded outer functions \cite[Proof of Theorem~2.1]{Du70}. 
Thus, by Theorem~\ref{T:GKZmodule}, 
there exists a character $\chi$ on $H^\infty$ such that
\begin{equation}\label{E:Hplinfun}
\Lambda(hf)=\chi(h)\Lambda(f) \qquad(f\in H^p,~h\in H^\infty).
\end{equation}

Let $c:=\Lambda(1)$ and $w:=\chi(u)$ 
(where $u$ denotes the function $u(z):=z$).
As $1$ is an outer function, we have $c\in\CC\setminus\{0\}$. 
Also,  for all $\lambda\in\CC\setminus\DD$, 
the function $(u-\lambda 1)$ is outer, 
so we have  $\Lambda(u-\lambda 1)\ne0$,
whence $\chi(u-\lambda 1)\ne0$ and  $w\ne\lambda$. 
In other words, $w\in\DD$.

To finish the proof, 
we show that $\Lambda(f)=cf(w)$ for all $f\in H^p$.
Given $f\in H^p$,  let us define $k(z):=(f(z)-f(w))/(z-w)$. 
Then $k\in H^p$ and  $f=f(w)1+(u-w1)k$. 
Applying $\Lambda$ to both sides of this 
last identity and using \eqref{E:Hplinfun},
we obtain 
\[
\Lambda(f)=f(w)\Lambda(1)+\chi(u-w1)\Lambda(k)=cf(w)+0,
\]
as desired.
\end{proof}

This leads to the following characterization of weighted composition operators.

\begin{theorem}\label{T:Hplinmap}
Let $0< p\le\infty$ and let $T:H^p\to\hol(\DD)$ be a linear map 
such that $(Tg)(z)\ne0$ for all outer functions $g\in H^p$ and all $z\in\DD$.
Then there exist holomorphic functions 
$\phi:\DD\to\DD$ and $\psi:\DD\to\CC\setminus\{0\}$ such that 
\[
Tf=\psi.(f\circ\phi)\qquad(f\in H^p).
\]
\end{theorem}

Note that continuity of $T$ is not assumed.

\begin{proof}
Define $\psi:=T1$. 
This is a  holomorphic function on $\DD$, 
and is nowhere zero because $1$ is outer.
Define $\phi:=(Tu)/\psi$, where $u(z):=z$. 
This too  is a holomorphic function  on $\DD$. 
For $z\in\DD$, the map $f\mapsto (Tf)(z)$ 
is a linear functional on $H^p$ that is non-zero on outer functions. 
By Theorem~\ref{T:Hplinfun}, 
there exist $w\in\DD$ and $c\in\CC\setminus\{0\}$ 
such that $(Tf)(z)=cf(w)$ for all $f\in H^p$. 
Taking $f:=1$, we see that $c=\psi(z)$.
Taking $f:=u$, we see that $w=\phi(z)$. 
Thus  $\phi(z)\in\DD$ and $(Tf)(z)=\psi(z)f(\phi(z))$ for all $f\in H^p$.
\end{proof}

We denote by $\aut(\DD)$  the group of holomorphic automorphisms of $\DD$,
namely the set of functions of the form $\phi(z):=c(z-w)/(1-\overline{w}z)$, 
where $w\in\DD$ and $c\in\TT$.

\begin{theorem}\label{T:Hplinop}
Let $1\le p\le\infty$ and let $T:H^p\to H^p$ be a surjective linear map 
such that $(Tg)(z)\ne0$ for all outer functions $g\in H^p$ and all $z\in\DD$.
Then $T$ is an invertible operator,
and there exist $\phi\in\aut(\DD)$ and $\psi\in H^\infty\cap (H^\infty)^{-1}$ 
such that 
\[
Tf=\psi.(f\circ\phi)\qquad(f\in H^p).
\]
\end{theorem}

\begin{proof}
By Theorem~\ref{T:Hplinmap}, 
there exist holomorphic maps $\phi:\DD\to\DD$ and $\psi:\DD\to\CC\setminus\{0\}$  
such that $Tf=\psi.(f\circ\phi)$ for all $f\in H^p$. 
As $T$ is surjective,
it is clear that $\phi$ is non-constant, 
so $\ker T=\{0\}$ and $T$ is  injective. 
By the closed graph theorem, $T$ is continuous.
Hence, by Banach's isomorphism theorem, $T$ is an invertible operator.
Finally, the invertibility of $T$ implies that 
$\phi$ is an automorphism of $\DD$ and
that both $\psi$ and $1/\psi$ are bounded: 
this is proved for example in \cite[Theorem~2.1]{Bo14},
and we shall give another proof in Theorem~\ref{T:linop} below.
\end{proof}

It is known that, if $1\le p\le\infty$ and $p\ne2$, 
then every surjective isometry $T:H^p\to H^p$ is of the form
\begin{equation}\label{E:isometry}
Tf=c .(\phi')^{1/p}.(f\circ\phi) 
\qquad(f\in H^p),
\end{equation}
where $\phi\in\aut(\DD)$ and $c$ is a unimodular constant.
This  was first established in the cases $p=1,\infty$, 
independently by Nagasawa \cite{Na59} 
and  by de Leeuw, Rudin and Wermer \cite{DRW60}. 
It was later extended to all $p\ne2$ by Forelli \cite{Fo64}. 

When $p=2$, there  are lots of surjective isometries 
other than those in \eqref{E:isometry}.  
For example, the unitary operator on $H^2$ that exchanges $1$ and $z$ 
while fixing all powers $z^k~(k\ge2)$
is clearly not of the type given in \eqref{E:isometry}. 
The next result adds an additional hypothesis 
that is sufficient to recover \eqref{E:isometry} in the case $p=2$.

\begin{theorem}\label{T:H2isometry}
Let  $T:H^2\to H^2$ be a surjective isometry 
such that $(Tg)(z)\ne0$ for all outer functions $g\in H^2$ and all $z\in\DD$.
Then there exist $\phi\in\aut(\DD)$ and a unimodular constant $c$ such that 
\[
Tf=c.(\phi')^{1/2}.(f\circ\phi)\qquad(f\in H^2).
\]
\end{theorem}

\begin{proof}
By Theorem~\ref{T:Hplinop}, 
there exist  $\phi\in\aut(\DD)$ and  $\psi\in H^\infty$ 
such that $Tf=\psi.(f\circ\phi)$ for all $f\in H^2$. 
Define $S:H^2\to H^2$ by $S f:=(\phi')^{1/2}.(f\circ\phi)$. 
Then $S$ is an isometry of $H^2$ onto itself, 
consequently so is $T\circ S^{-1}$. 
A simple calculation shows that $(T\circ S^{-1})$ has the form 
$(T\circ S^{-1}) f= h.f$ for all $f\in H^2$, where $h\in\hol(\DD)$. 
The fact that both $T\circ S^{-1}$ and its inverse are contractions on $H^2$ 
implies that both $h$ and $1/h$ belong to $H^\infty$ 
with $\|h\|_\infty=\|1/h\|_\infty=1$. 
This implies that $h\equiv c$, a unimodular constant. 
Hence $T=cS$, as desired.
\end{proof}

\section{Extensions to other function spaces}\label{S:otherspaces}

The proof of Theorem~\ref{T:Hplinfun} uses special properties of the Hardy spaces,
and does not seem to extend easily to other families of spaces. 
However, if we are willing to assume the continuity of the  linear maps involved, 
then the theorems of the previous section
do indeed extend to a wide variety of other spaces, 
albeit with slightly different proofs.

In what follows, 
we shall consider a Banach space $X\subset\hol(\DD)$ with the following properties:
\begin{itemize}
\setlength{\itemindent}{25pt}
\item[(X1)] for each $w\in\DD$, the evaluation map $f\mapsto f(w):X\to\CC$ is continuous;
\item[(X2)] $X$ contains the polynomials and they form a dense subspace of $X$;
\item[(X3)] $X$ is shift-invariant: $f\in X\Rightarrow zf\in X$.
\end{itemize}
We write $\cM(X)$ for the \emph{multiplier algebra} of $X$, namely
\begin{align*}
\cM(X)&:=\{h\in\hol(\DD): hf\in X \text{~for all~} f\in X\},\\
\|h\|_{\cM(X)}&:=\sup\{\|hf\|_X:\|f\|_X\le1\}.
\end{align*}
Using property (X1) above, it is not hard to see that
$\cM(X)$ can be identified with a closed subalgebra of 
the algebra of all bounded linear operators on $X$, so it is a Banach algebra.
For each $w\in\DD$, 
the evaluation functional $h\mapsto h(w)$ is a character on $\cM(X)$,
so $|h(w)|\le\|h\|_{\cM(X)}$. 
It follows that $\cM(X)\subset H^\infty$.
As $X$ contains the constants,  we also have $\cM(X)\subset X$.

We also consider a subset $Y$ of $X$ with the following properties:
\begin{itemize}
\setlength{\itemindent}{25pt}
\item[(Y1)] if $g\in X$ and $0<\inf_\DD|g|\le\sup_\DD|g|<\infty$, then $g\in Y$;
\item[(Y2)] if $g(z):=z-\lambda$ where $\lambda\in\TT$, then $g\in Y$.
\end{itemize}
 
Examples of spaces $X$ satisfying (X1)--(X3) include:
\begin{itemize}
\item the Hardy spaces $H^p~(1\le p<\infty)$;
\item the Bergman spaces $A^p~(1\le p<\infty)$;
\item the holomorphic Besov spaces $B_p~(1\le p<\infty)$;
\item the weighted Dirichlet spaces $\cD_\alpha~(0\le\alpha\le1)$;
\item the holomorphic Sobolev spaces $S^p:=\{f:f'\in H^p\}~ (1\le p<\infty)$;
\item the disk algebra $A(\DD)$;
\item the little Bloch space $\cB_0$;
\item the  space $\vmoa$ of functions of vanishing mean oscillation;
\item the de Branges--Rovnyak spaces $\cH(b)$, for non-extreme points $b$  in the unit ball of $H^\infty$.
\end{itemize}

Examples of sets $Y$ satisfying (Y1)--(Y2) include:
\begin{itemize}
\item the nowhere-zero functions in $X$;
\item the outer functions in $X$;
\item the cyclic functions  for the shift, if $X=H^p,A^p, B_p,\cD_\alpha,\cB_0$ or $\vmoa$.
\end{itemize}

For background on these various function spaces, we refer to the books \cite{FM15} and \cite{Zh07}.

\begin{theorem}\label{T:linfun}
Let $X\subset\hol(\DD)$ be a Banach space satisfying (X1)--(X3) above.
Let $Y\subset X$ be a set satisfying (Y1)--(Y2) above.
Let $\Lambda:X\to\CC$ be a continuous linear functional such that 
$\Lambda(g)\ne0$ for all $g\in Y$.
Then there exist $w\in\DD$ and  $c\in\CC\setminus\{0\}$ such that 
\[
\Lambda(f)=cf(w) \qquad(f\in X).
\]
\end{theorem}

\begin{proof}
Since $1\in Y$, we have $c:=\Lambda(1)\ne0$. 
Replacing $\Lambda$ by $\Lambda/c$, we may suppose that $\Lambda(1)=1$.
If $h$ is invertible in $\cM(X)$, 
then both $h$ and $1/h$ belong to $H^\infty\cap X$, 
so $h\in Y$, and $\Lambda(h)\ne0$.
By Theorem~\ref{T:GKZ}, $\Lambda$ is a character on $\cM(X)$. 
As $X$ is shift-invariant, we have $u\in\cM(X)$ (where $u(z):=z$).
Set $w:=\Lambda(u)$. 
Then $\Lambda(p)=p(w)$ for all polynomials $p$.
If $|\lambda|\ge1$ then $u-\lambda\in Y$, 
so $\Lambda(u-\lambda)\ne0$. 
Consequently $w\in\DD$,
and the evaluation functional  $f\mapsto f(w)$ is continuous on $X$. 
As polynomials are dense in $X$, 
we conclude that $\Lambda(f)=f(w)$ for all $f\in X$.
\end{proof}

In the next theorem, we endow $\hol(\DD)$ with its usual Fr\'echet-space topology.

\begin{theorem}\label{T:linmap}
Let $X\subset\hol(\DD)$ be a Banach space satisfying (X1)--(X3) above.
Let $Y\subset X$ be a set satisfying (Y1)--(Y2) above.
Let $T:X\to \hol(\DD)$ be a continuous linear map such that 
$Tg(z)\ne0$ for all $g\in Y$ and all $z\in\DD$. 
Then there exist holomorphic functions 
$\phi:\DD\to\DD$ and $\psi:\DD\to\CC\setminus\{0\}$ such that
\[
Tf=\psi.(f\circ\phi) \qquad (f\in X).
\]
\end{theorem}

\begin{proof}
This follows from Theorem~\ref{T:linfun} in just the same way that 
Theorem~\ref{T:Hplinmap} was deduced from Theorem~\ref{T:Hplinfun}.
\end{proof}

\begin{theorem}\label{T:linop}
Let $X\subset\hol(\DD)$ be a Banach space satisfying (X1)--(X3) above, 
and let $Y\subset X$ be a set satisfying (Y1)--(Y2) above.
Suppose in addition
that $X$ is M\"obius-invariant:
\[
f\in X, ~\phi\in\aut(\DD)\Rightarrow f\circ\phi\in X.
\]
Let $T:X\to X$ be a continuous linear surjection such that 
$Tg(z)\ne0$ for all $g\in Y$ and all $z\in\DD$. 
Then $T$ is invertible, 
and there exist $\phi\in \aut(\DD)$ and $\psi\in \cM(X)\cap \cM(X)^{-1}$ such that
\[
Tf=\psi.(f\circ\phi) \qquad (f\in X).
\]
\end{theorem}

For the proof, it is convenient to establish a lemma. 
As usual, we write $u(z):=z$.

\begin{lemma}
The spectrum of $u$ in $\cM(X)$ is equal to $\overline{\DD}$.
\end{lemma}

\begin{proof}
For each $w\in\DD$, 
the map $h\mapsto h(w)$ is a character on $\cM(X)$,
so $w$ belongs to the spectrum of $u$. 
As the spectrum of $u$ is compact, it must contain $\overline{\DD}$.

By (X3) and M\"obius-invariance, 
 $(u-w)/(1-\overline{w}u)$ is a multiplier of $X$ for each $w\in\DD$.
Subtracting and multiplying by suitable constants,
we see that $(1-\overline{w}u)^{-1}\in\cM(X)$ for all $w\in\DD$. 
Thus the spectrum of $u$ is contained within  $\overline{\DD}$. 
\end{proof}

\begin{proof}[of Theorem~\ref{T:linop}]
By the closed graph theorem, 
the inclusion map from $X$ into $\hol(\DD)$ is continuous,
so $T$ may be regarded as a continuous map $:X\to\hol(\DD)$.
By Theorem~\ref{T:linmap}, 
there are holomorphic functions 
$\phi:\DD\to\DD$ and $\psi:\DD\to\CC\setminus\{0\}$  such that
$Tf=\psi.(f\circ\phi)$ for all $f\in X$. 
As $T$ is surjective, $\phi$ is non-constant, so $T$ is also injective. 
By Banach's isomorphism theorem, $T$ is invertible, 
and there is a constant $C$ such that
$\|f\|_X\le C\|Tf\|_X$ for all $f\in X$.

We next show that $\phi$ is an automorphism of $\DD$.
As $u\psi=z(T1)\in X$, 
there exists $\theta\in X$ such that $u\psi=T\theta$. 
Then $u\psi=\psi.(\theta\circ\phi)$, 
whence $(\theta\circ\phi)(z)=z$ for all $z\in\DD$. 
Applying $\phi$ to both sides, 
we have $(\phi\circ\theta)(w)=w$ for all $w\in\phi(\DD)$. 
We claim that $\theta(\DD)\subset\DD$. 
If so, then $\phi\circ\theta$ is well-defined on $\DD$, 
and by the identity principle $(\phi\circ\theta)(w)=w$ for all $w\in\DD$, 
showing that $\phi$ is indeed an automorphism.

To justify the claim, 
observe that, for each $n\ge1$, as $u^n\psi=z^n(T1)\in X$,  
there exists $\theta_n\in X$ such that $u^n\psi=T\theta_n$. 
Then we have  $u^n\psi=\psi.(\theta_n\circ\phi)$, 
so $\theta_n\circ\phi=u^n=\theta^n\circ \phi$ 
and  $\theta_n=\theta^n$. 
Thus $\theta^n\in X$ for all $n$ and $T(\theta^n)=u^n\psi$.
By the Banach isomorphism theorem inequality, it follows that
\[
\|\theta^n\|_X\le C\|u^n\psi\|_X=C\|u^n\|_{\cM(X)}\|\psi\|_X.
\]
Taking $n$th roots and letting $n\to\infty$, we obtain that 
\[
\limsup_{n\to\infty}\|\theta^n\|_X^{1/n}\le \limsup_{n\to\infty}\|u^n\|_{\cM(X)}^{1/n}.
\]
By the lemma and the spectral radius formula, 
the right-hand side equals $1$.
As point evaluations at points of $\DD$ are continuous on $X$, 
the left-hand side is bounded below by $|\theta(w)|$ for each $w\in\DD$.
It follows that $|\theta(w)|\le 1$ for all $w\in\DD$.
As $\theta$ is non-constant, 
the maximum principle implies that $|\theta(w)|<1$ for all $w\in\DD$.
This proves the claim.

Finally, we show that $\psi\in\cM(X)\cap\cM(X)^{-1}$. 
Let $C_\phi(f):=f\circ \phi$. 
Using the M\"obius-invariance of $X$ and the closed graph theorem,
we see that $C_\phi$ is an isomorphism of $X$ onto itself, 
and consequently so too is $T\circ C_\phi^{-1}$. 
But $T\circ C_\phi^{-1}$ is just the multiplication operator $f\mapsto \psi.f$, 
so both $\psi$ and $1/\psi$ are multipliers of $X$.
\end{proof}

\begin{remarks*}
(i) All the examples of spaces $X$ satisfying (X1)--(X3) listed earlier are M\"obius--invariant, 
with the exception of the de Branges--Rovnyak spaces $\cH(b)$.

(ii) Neither the Bloch space $\cB$  
nor the space $\bmoa$ of holomorphic functions of bounded mean oscillation satisfies (X2), 
because polynomials are not dense. 
However, these spaces are the biduals of $\cB_0$ and $\vmoa$ respectively, 
and it is not hard to see that Theorems~\ref{T:linfun}--\ref{T:linop} remain true for them 
provided that one replaces norm-continuity by weak*-continuity throughout.

(iii) If one tries to apply Theorem~\ref{T:GKZmodule} directly to the spaces in \S\ref{S:otherspaces}, 
then one runs up against some interesting problems. 
For example, can every function $X$ can be factorized as a multiplier times a cyclic function? 
Can every function in $X$ be written as the quotient of two multipliers of $X$?
In the case when $X$ is the Dirichlet space, 
this last question was posed as a problem at the end of \S3  in \cite{Ro06}, 
and as far as we know it is still open. 
\end{remarks*}

\begin{acknowledgements}
This paper arose from a question of Andrew Mullhaupt, relayed to us by Malik Younsi,
as to whether a result along lines of Theorem~\ref{T:H2isometry} might be true.
We are grateful to both of them for providing this stimulation.
We also thank Ken Davidson for suggesting a simplification of the proof of Theorem~\ref{T:Hplinfun}.
\end{acknowledgements}

\affiliationone{
Thomas Ransford and Javad Mashreghi\\
D\'{e}partement de math\'{e}matiques et de statistique\\
Universit\'{e} Laval\\Qu\'{e}bec (QC) \\ Can\-a\-da G1V 0A6\\
\email{javad.mashreghi@mat.ulaval.ca\\
thomas.ransford@mat.ulaval.ca}}

\end{document}